\documentclass[letter,12pt]{article}
\usepackage[T1]{fontenc}
\usepackage{verbatim}
\usepackage{amsmath,amsthm}
\usepackage{xspace}
\usepackage{pifont}
\usepackage{graphicx}
\usepackage{amssymb}
\usepackage{epic, eepic}
\usepackage{dsfont}
\usepackage{amssymb}
\usepackage{makeidx}
\usepackage{mathrsfs}

\usepackage{amsmath,afterpage}
\usepackage{epsf}
\usepackage{graphics,color}

\newcommand{\zz}[1]{\mathbb{#1}}
\def\q{\quad}

\def\0{\mathbf{0}}

\def\eps{\varepsilon}

\def\rr{\rightarrow}

\def\al{\alpha}

\def \< {\langle}
\def \> {\rangle}

\def\ol{\overline}
\def\ul{\underline}

\def\beqa{\begin{eqnarray}}
\def\eeqa{\end{eqnarray}}
\def\beqas{\begin{eqnarray*}}
\def\eeqas{\end{eqnarray*}}

\newtheorem{theorem}{Theorem}[section]
\newtheorem{lemma}[theorem]{Lemma}
\newtheorem{prop}[theorem]{Proposition}
\newtheorem{proposition}[theorem]{Proposition}

\newtheorem{cor}[theorem]{Corollary}

\newtheorem{remark}[theorem]{Remark}

\numberwithin{equation}{section}
\newcommand{\hatd}[1]{{}}

\newcommand{\bd}{\begin{displaymath}}
\newcommand{\ed}{\end{displaymath}}
\newcommand{\be}{\begin{equation}}
\newcommand{\ee}{\end{equation}}
\newcommand{\bq}{\begin{eqnarray}}
\newcommand{\eq}{\end{eqnarray}}
\newcommand{\bn}{\begin{eqnarray*}}
\newcommand{\en}{\end{eqnarray*}}

\newcommand{\W}{\mathcal{W}}

\def\wt{\widetilde}

\newcommand{\indic}{\mathds1}

\def\const{\rho}				       	   

\begin{document}

\title{On the Maximal Displacement of Subcritical Branching Random Walks\footnote{Research partially supported by GRF 606010 and 607013 of the HKSAR and the HKUST IAS Postdoctoral Fellowship.} }
\author{Eyal Neuman\footnote{Eyal Neuman would like to thank HKUST where most of the research was carried out.}, Xinghua Zheng}

\date{\today}

\maketitle

\begin{abstract}
We study the maximal displacement of a one dimensional subcritical branching random walk initiated by a single particle at the origin. For each $n\in\zz{N},$ let $M_{n}$ be the rightmost position reached by the branching random walk up to generation $n$.
Under the assumption that the offspring distribution has a finite third moment and the jump distribution has mean zero and
a {finite probability generating function}, we show that there exists $\const>1$ such that the function
\[
g(c,n):=\const^{cn} P(M_{n}\geq cn), \q \mbox{for each }c>0 \mbox{ and } n\in\zz{N},
\]
satisfies the following properties: there exist $0<\ul{\delta}\leq \ol{\delta} < {\infty}$ such that
if $c<\ul{\delta}$, then
\[
0<\liminf_{n\rr\infty} g (c,n)\leq \limsup_{n\rr\infty} g (c,n) {\leq 1},
\]
while
 if $c>\ol{\delta}$, then
\[
\lim_{n\rr\infty} g (c,n)=0.
\]
Moreover, if the jump distribution has a finite right range $R$, then $\ol{\delta} < R$. If furthermore the jump distribution is ``nearly right-continuous'',
then there exists $\kappa\in (0,1]$ such that $\lim_{n\rr \infty}g(c,n)=\kappa$ for all $c<\ul{\delta}$. We also show that the tail distribution of $M:=\sup_{n\geq 0}M_{n}$, namely, the rightmost position ever reached by the branching random walk, has a similar exponential decay
(without the cutoff at $\ul{\delta}$).
Finally, by duality, these results imply that the maximal displacement of supercritical branching random walks conditional on extinction has a similar tail behavior.
\end{abstract}

\section{Introduction and Main results}

Extreme values of spatial branching systems have been extensively studied over the past decades. Results on the asymptotic law for the maxima of branching Brownian motion trace back to Sawyer and Fleischman \cite{S-F78} and Lalley and Sellke \cite{Lalley-Sellke1987}. The work on the strong law of large numbers for the maxima of branching random walk trace back to Hammersley \cite{Hammersley}, Kingman \cite{Kingman}, Biggins \cite{Biggins} and Bramson \cite{Bramson:1978}, however results on the tail behavior of the maxima only appeared in resent years.

In branching processes we distinct among three subclasses according to the mean number of offspring, which we denote by $m$. The position of the rightmost particle at a specific generation for supercritical branching random walk (that is $m>1$)  was extensively studied in recent years, see for example \cite{Hu-Shi09,Bachmann:2000,Bramson-Z09,Aidekon2013,Bramson-Ding-Z:2014} and references therein. In particular  Aidekon proved in \cite{Aidekon2013}
that the centred law of the maximal displacement  converges to a random shift of the Gumbel distribution (see also \cite{Bramson-Ding-Z:2014}).

The case where the offspring distribution is critical, that is $m=1$, was recently studied by Lalley and Shao in \cite{LS15}. Let $M$ is the rightmost point ever reached by the branching random walk. It was proved in \cite{LS15} that when the jump distribution has mean $0$, then under some moment assumptions,
\bn
P\big(M\geq x\big) \sim \frac{\alpha}{x^{2}}, \textrm{ as } x \rr\infty .
\en
Here $\alpha$ is a constant which depends on the standard deviations of the jump distribution and the offspring distribution.

This paper is devoted to the study of the subcritical case, i.e. $m<1$.
Sawyer and Fleischman in \cite{S-F78} studied the law of the rightmost position ever reached by a subcritical branching Brownian motion. Using analytic methods it was proved in \cite{S-F78} that $h(x)=P(M>x)$ satisfies an ordinary differential equation
\bn
\frac{1}{2}h''(x)=-\phi\big(h(x)\big),
\en
where $\phi(\cdot)$ is the probability {generating} function of the offspring distribution.
By solving this equation it was shown in \cite{S-F78} that
\bq \label{Saw-Fl-res}
\lim_{x\rr\infty}\frac{P(M>x)}{(1-m)s(x)e^{-\sqrt{2(1-m)}\, x}}=1,
\eq
where $s(x)$ is a bounded positive function. One of the goals of this work is to
establish a similar exponential decay for the maximal displacement of subcritical branching random walk.
In fact, when the jump distribution has a finite right range and is ``nearly right-continuous'', our result is sharper than \eqref{Saw-Fl-res} in the sense that $s(x)$ can be replaced with a constant, see Theorem \ref{theorem-max-dis} below.

The other main focus of this paper is the maximal displacement of the local time of branching random walk.
The motivation for this comes from the study of population models, where sharp bounds on  local time processes are often key elements in the proofs. For example, in \cite{L-P-Z-2014}, a phase transition for the spatial measure-valued susceptible-infected-recovered (SIR) epidemic process is established. A key ingredient in the proof
is the growth rate of the support of the local time process (see discussion in Section 1.2 of \cite{L-P-Z-2014}).
The propagation of the support of the local time process was studied in \cite{iscoe} and \cite{pinsky}, for critical and supercritical super-Brownian motions which are scaling limits of critical and supercritical branching random walks. Results on the limiting measure-valued processes in most cases are not enough for  research on discrete particle systems. The reason is that after taking scaling limits, mass of the discrete process of lower order than the scaling dimension vanishes
(see for example the discussion after Proposition 23 in \cite{LS15}). In order to study fine properties of spatial discrete particle models, one needs to get more precise bounds on the support of discrete local time processes themselves.
Observe that the maxima of the support of the local time at a generation $n$ is nothing but
$M_{n}$, the rightmost position reached by the branching random walk up to generation~$n$. While the maximal displacement \emph{at} generation $n$ was studied for critical and supercritical branching particle systems, we could not find any previous work on the study of $M_{n}$ for discrete branching systems.

Before we state our main results, we define more carefully the branching random walk that we study.

\textbf{The model:} We consider a discrete time branching random walk on $\zz{Z}$ that lives on a probability space
$(\Omega, \mathcal F)$. In each generation, particles first jump (independently from each other) according to
a distribution $F_{RW}:= \{a_{y}\}_{y\in {\zz{Z}}}$, which has mean zero and finite variance,
and then each particle reproduces independently, according to an offspring distribution $F^{}_{B}:=\{p^{}_{k}\}_{k\geq 0}$,
which has expectation $m\in(0,1)$, variance $ \sigma^2 $ and
a finite third moment.

\begin{remark}\label{rmk:branch_jump_order}
Observe that under this model, particles jump first and then reproduce, just as in \cite{LS15}. In many other studies (see e.g. \cite{Bramson:1978, pinsky, L-Z2010, L-Z2011}),
the order is reversed, namely, particles reproduce first and then the offspring particles jump. The two ways of definition does affect the value of the maximal displacement,
see Remark 1 in \cite{LS15} for a simple example. However, the tail behaviors of the maximal displacement are the same up to a multiplicative constant, see Remark~3 therein and
equation
\eqref{Q-N}
below.
\end{remark}

To formulate things more precisely, below we borrow some notation from \cite{Aidekon2013}. We assume that the branching random walk is initiated by
a single particle at the origin. Let $\mathcal{T}$ be the genealogical Galton-Watson tree of this system.
For each vertex $v\in \mathcal T$, we denote by $|v|$ its generation, and $X(v)$ its position on the real line.
For each $n\geq 1$, let
\[
 Z_n = \#\{v\in \mathcal T: \ |v|= n\}
\]
be the number of particles at generation $n$.
The collection of positions $X:=\{X (v); v \in \mathcal{T} \}$ defines our branching system.
We study the
tail behaviors of the maximal displacements of $X$ up to generation $n$ and over all generations, namely,
\be \label{M-n-def}
M_{n} = \max_{v\in \mathcal T, \ |v|\leq n}X(v), \q\mbox{and}\q M = \sup_{n\geq 0}M_{n}.
\ee
Define
\bn
u_{n}^{}(x)=P^{}(M_{n}^{}\geq x)\q\mbox{and}\q u(x)=P^{}(M^{}\geq x), \q\mbox{for all } x \in \zz{Z}.
\en
Clearly $u_n(x)\leq u(x)$ for all $n$ and $x$, both are decreasing in $x$, and since $X$ dies out almost surely, we have
\[
\lim_{x\rr\infty} u(x) = 0.
\]
Before we present our first result, we introduce the following notation.
\paragraph{Notation.}
We denote by $W=\{W_{n}\}_{n \geq 0}$ a random walk on $\zz{Z}$ with the following law
\bn
P(W_{n+1}-W_{n}=k\ |\ W_{n},W_{n-1},...)=a_{k}, \ k \in \zz{Z}.
\en
For each $y\geq 0$,  define
\bq \label{tau-y}
\tau_{y}=\min\{n\geq 0: W_{n}\geq y\}.
\eq
We will use
$P^{x}, \ E^{x}$
to denote the probability and expectation under the distribution  of $\{W_{n}\}_{n \geq 0}$ with $W_{0}=~x$, and omit the superscript when $x=0$ (and when there is no confusion).

Further define $K(\theta)$ to be the probability generating function of $W_1$:
\be\label{pgf}
K(\theta)=E\big(\theta^{W_{1}}\big), \q\mbox{for } \theta\geq 1.
\ee
Throughout this work we assume that
\be \label{pgf-cond}
K(\theta) <\infty, \q \textrm{for all } \theta \geq 1.
\ee
We have $K(1)=1$ and $\lim_{\theta\rr\infty }K(\theta)=\infty$. Moreover, it is easy to verify that the condition $E(W_{1})=0$ ensures that $K'(\theta)>0$ for all $\theta>1$, hence for every $\gamma>1$,  there exists a unique solution to $K(s)=\gamma$ between $(1,\infty)$. Denote such a unique solution by $\const(\gamma)$.

Finally, we say that $W$ has a finite right range  $R(>0)$ if  $a_R>0$ and $a_{k}=0$ for all $k>R$. We further say that
$W$ is \textit{nearly right-continuous} if $a_{i}>0$ for all $i=1,\ldots,R$.

Now we are ready to present our first result.
The following theorem gives  the asymptotic behavior of $u(x)$ as $x\rr\infty$, which is the branching random walk analog to the result of Sawyer and Fleischman in \cite{S-F78}.
\begin{theorem} \label{theorem-max-dis}
Assume that the offspring distribution $F_{B}$ has a finite third moment and that the jump distribution $F_{RW}$ {satisfies (\ref{pgf-cond})}, then
\begin{itemize}
\item[\bf{(a)}]
$\ell(n) := \const\big(m^{-1}\big)^{n}u(n)$
satisfies that
\[
0<\liminf_{n \rr\infty }\ell(n)\leq \limsup_{n \rr\infty }\ell(n) \leq 1.
\]
\item[\bf{(b)}] Moreover,  if $W$ {has a finite right range $R$} and is nearly right-continuous, then there exists $\kappa\in(0,1]$ such that
\bn
\lim_{n\rr\infty} \ell(n) = \kappa.
\en
\item[\bf{(c)}] In either case,
\bq \label{lim-C}
\log \const(m^{-1})=-\lim_{n\rr\infty}\frac{\log E\big(m^{\tau_{n}}\big)}{n}.
\eq
\end{itemize}
\end{theorem}
\begin{remark}  \label{rem-special-cases}
In the case of nearest neighbor branching random walk, namely, when $a_{1}=a_{-1}=1/2$, we have $\const(m^{-1})=(1+\sqrt{1-m^{2}})/m$,
and it is easy to verify, by conditioning on the first step, that $(1+\sqrt{1-m^{2}})/m=1/E\big(m^{\tau_{1}}\big)\equiv 1/\left(E\big(m^{\tau_{n}}\big)\right)^{1/n}$ for all $n\in\zz{N}$.
\end{remark}

\begin{remark}  \label{rem-special-cases2}
In the case where the jump distribution is such that $a_{-2}=1/2$ and $a_{2}=1/2$, we have $u(2n)=u(2n-1)$. It is then easy to see that $\lim_{n\rr\infty} \ell(n)$ does not exist.
\end{remark}

\begin{remark}\label{rmk:RW_exp_tail}
The assumption \eqref{pgf-cond} can be weakened to be $K(\theta)<\infty$ for some $\theta > \rho(m^{-1}).$ On the other hand, since $P(M \geq n) \geq (1-p_0) P(W_1 \geq n)$,
an exponential decaying tail of $W_1$ is necessary for that of $M$.
\end{remark}

We now move on to our second main result, concerning
the maximal displacement of~$X$ up to generation $n$. This result also provides insights into why $M$ has such an exponential decay.

\begin{theorem} \label{theorem-max-dis-upto}
Assume that the offspring distribution $F_{B}$ has a finite third moment and that the jump distribution $F_{RW}$ satisfies (\ref{pgf-cond}). For $c\geq 0$ and $n\in\zz{N}$, define $g(c, n) = \const\big(m^{-1}\big)^{cn}P(M_{n}\geq cn)$.
There exist $0<\ul{\delta}\leq\ol{\delta}< {\infty}$ such that
\begin{itemize}
\item[\bf{(a)}] For every $c\in(0,\ul{\delta})$,
we have
    \[
     0<\liminf_{n\rr \infty}g(c,n)\leq \limsup_{n\rr \infty}g(c,n) \leq 1.
     \]
\item[\bf{(b)}] Moreover, if $W$ {has a finite right range $R$} and is nearly right-continuous, then
\bn
\lim_{n\rr\infty} g(c,n) = \kappa, \q \mbox{for all } c\in(0,\ul{\delta}),
\en
where $\kappa$ is the same constant that appears in Theorem \ref{theorem-max-dis}(b).
\item[\bf{(c)}] For every $c>\ol{\delta}$, we have $\lim_{n\rr \infty}g(c,n) = 0$. Moreover, if $W$ has a finite right range $R$, then $\ol{\delta} < R$.
\end{itemize}
\end{theorem}

A few remarks about the theorem are in order.

\begin{remark}\label{rmk:M_Mn}
We prove parts (a) and (b) by showing that there exists $a>0$ such that $P(M_{an}\geq n\ |\ M\geq n)\to 1$, see Corollary \ref{cor:M_Mn} below. Roughly speaking, in order for $M\geq n$, either the branching random walk spreads out abnormally in a linear speed, or it spreads out like an ordinary random walk in which case the process has to survive $O(n^2)$ generations. Corollary \ref{cor:M_Mn} indicates that the first possibility dominates, and the exponential decay of $M$ is due to both subcriticality of the branching process and the large deviation of the random walk.
\end{remark}

\begin{remark}\label{rmk:part_c}
About part (c), in the case where $W$ has a finite right range $R$,  one has  $P(M_{n}\geq cn)\equiv 0$ for all $c>R$, hence $g(c,n)\to 0$ trivially holds.
The significance of part (c) in this case lies in that there exists $c<R$ such that $\const\big(m^{-1}\big)^{cn}P(M_{n}\geq cn)\to~0$.
\end{remark}

\begin{remark}\label{rmk:jump}
We conjecture that $\ul{\delta}=\ol{\delta}:=\delta^*$, in other words, the function
\[
\psi(c):= \liminf_{n\rr\infty} g(c,n)
\]
is  positive for $c<\delta^*$ and equals $0$ for $c>\delta^*$. Such a strange phase transition (if our conjecture were true) has to do with the \emph{local small deviations} of the first passage times of the associated random walk. To the end of Section \ref{sec-Heuristics}, we prove this conjecture for a special subcritical branching random walk and we also derive an exact {local small deviation} result for the first passage times of the nearest neighbor random walk.
\end{remark}

Before we state the next result, we recall the duality principle which states that a supercritical branching process conditional on extinction has the same distribution as its dual subcritical process (see, for example, Theorem 3 in Chapter I.12 in \cite{Athreya-1972}). More specifically,
let $\overline Z= \{\overline Z_{n}\}_{n\geq 0}$ be a Galton-Watson process with $\ol{Z}_0=1$ and an offspring distribution $\overline F_{B} = \{\ol{p}_i\}_{i\geq 0}$ which has mean $m>1$ and $\ol{p}_0>0$.
Define
\bn
B^{}=\{\omega:\overline Z^{}_{n}(\omega)=0, \textrm{ for some } n\geq 1\}
\en
be the event of extinction, and let $q=P(B)\in (0,1)$.
Then the duality principle says that the process $\{\overline Z_{n}\}_{n\geq 0}$
conditional on the event $B$ has the same distribution as a subcritical Galton-Watson branching process $\{ \wt Z_{n}\}_{n\geq 0}$  with $\wt Z^{}_{0}=1$
and
\bn
E^{}\big(\theta^{\wt Z^{}_{1}}\big)  \equiv  \frac {\ol{f}(\theta q)}{q^{}}, \q \theta \in (0,1),
\en
where $\bar f^{}$ denotes the probability generating function of $\overline F_{B}$.

The duality principle allows us to extend Theorems \ref{theorem-max-dis} and \ref{theorem-max-dis-upto} to the supercritical case and obtain  analogous results about the maximal displacement of supercritical branching random walk conditional on extinction. We state the results below without giving the proofs.

To be more specific, suppose that $\overline X$ is a branching random walk which satisfies the same assumptions as $X$ except that it has an offspring distribution $\overline F_{B}$ as above.
Define $\overline M$ to be the maximal displacement of $\overline X$ over all generations as in (\ref{M-n-def}), and let
\[
 \ol{u}(n)=P^{}(\ol{M}^{}\geq n), \q\mbox{for all } n \in \zz{Z}.
\]
 Denote by $\bar f'$  the first derivative of~$\bar f$.

\begin{prop} \label{prop--sub-max-dis}
Assume that the offspring distribution $\overline F_{B}$ has a finite third moment and that the jump distribution $F_{RW}$ {satisfies (\ref{pgf-cond})}. Denote  $\ol{\rho} = \rho\big((\bar f'(q))^{-1}\big)$.
 Then we have
\begin{itemize}
\item[\bf{(a)}] $\ol{\ell}(n) := \ol{\rho}^{n}\big(\bar u(n)-(1-q)\big)$
satisfies that
\[
0<\liminf_{n\rr \infty} \ol{\ell}(n)\leq\limsup_{n\rr \infty} \ol{\ell}(n)\leq q.
\]
\item[\bf{(b)}] Moreover, if $W$ {has a finite right range $R$} and is nearly right-continuous, then there exists $\kappa'\in(0,q]$ such that
\bn
\lim_{n\rr\infty} \ol{\ell}(n) = \kappa'.
\en
\end{itemize}
\end{prop}

\paragraph{Organization of the paper:}
The rest of this paper is devoted to the proofs of
Theorems \ref{theorem-max-dis} and \ref{theorem-max-dis-upto}.
In Section \ref{sec-Heuristics} we give  heuristic explanations to Theorems~\ref{theorem-max-dis} and~\ref{theorem-max-dis-upto}
by considering a special subcritical branching random walk. We prove the conjecture in Remark \ref{rmk:jump} in this case,
and as a by-product, we establish a {local small deviation} result for the first passage times of the nearest neighbor random walk. Sections \ref{sec-proof-max-dis}--\ref{sec-proof-theorem-max-dis} are devoted to the proof of Theorem \ref{theorem-max-dis}.
In Section \ref{sec-proof-max-dis} we prove that $\lim_{n\rr\infty} (-\log u(n)/n)$ exists and is positive.
In Section \ref{sec-f-c} we derive a discrete Feynman-Kac formula for~$u(n)$.
Equipped with these tools we prove Theorem \ref{theorem-max-dis} in Section \ref{sec-proof-theorem-max-dis}.
Finally, we prove Theorem \ref{theorem-max-dis-upto} in Section \ref{Sec-pf-upto}.

\section{Heuristics for Theorems \ref{theorem-max-dis} and \ref{theorem-max-dis-upto}} \label{sec-Heuristics}
In this section we give  heuristic  interpretations to Theorems \ref{theorem-max-dis} and \ref{theorem-max-dis-upto} for a special branching random walk.

Let $X^{s}$ be a branching random walk such that in every step, each particle produces either no offspring with probability $p_{0}$, or one offspring with probability $1-p_{0}$. We further assume that the single particle in this model performs nearest neighbor random walk. Let $M^{s}_n$ and $M^{s}$ be the maximal displacements of $X^{s}$ up to generation~$n$ and over all generations as in (\ref{M-n-def}).
From the Markov property, we have for every $n\geq 1$,
\bn
P\big(M^{s}\geq n\big)
&=& P\big(M^{s} \geq 1\big)\ P\big(M^{s}\geq n|M^{s}\geq1\big)  \nonumber \\
&=&P\big(M^{s}\geq 1\big)^{n},
\en
in other words, $M^{s}$ has a geometric law with parameter $P\big(M^{s}<1\big)$. On the other hand, we have
\[
P\big(M^{s}\geq1\big)
=\sum_{j\geq 1}(1-p_{0})^{j}P\big(\tau_{1}=j)
=E\big( (1-p_{0})^{\tau_{1}} \big),
\]
where $\tau_{\cdot}$ is the first passage time defined in \eqref{tau-y}.
Furthermore, as we pointed out in Remark \ref{rem-special-cases},
\[
\const =\const\big((1-p_{0})^{-1}\big)= \frac{1}{E\big( (1-p_{0})^{\tau_{1}} \big)}.
\]
Hence
\be\label{eqn:exp_decay_nnbrw}
\const^{n}P\big(M^{s}\geq n\big)\equiv 1,
\ee
which agrees with Theorem \ref{theorem-max-dis}.

On the other hand, to see how Theorem  \ref{theorem-max-dis-upto}  (a)\&(b) follow from Theorem \ref{theorem-max-dis} for~$X^{s}$, we first observe that for any $a>1$,
\bq \label{j2}
P\big(M^{s} \geq n\big)  &=& P\big(M^{s}_{an} \geq n\big) +P\big(M^{s}_{an} < n, \ M^{s}\geq n\big) \nonumber  \\
&=& P\big(M^{s}_{an} \geq n\big) +\sum_{ j>  an}(1-p_{0})^{j}P\big(\tau_{n}=j\big).
\eq
By a simple calculation we get
that for $a>-\frac{\log \const }{\log(1-p_{0})} (>1)$,
\[
\const^{n}\sum_{j> an}(1-p_{0})^{j}P\big(\tau_{n}=j\big) \leq \const^{n} \frac{(1-p_{0})^{an}}{p_{0}}\rr 0,   \textrm{ as } n\rr\infty.
\]
It follows from \eqref{eqn:exp_decay_nnbrw} and (\ref{j2}) that
\[
\lim_{n\rr\infty} \const^{n} P\big(M^{s}_{a n}>n\big) = 1 .
\]
Define
\be\label{a:upper_bd}
\ol{a}:=\inf\{a: \lim_{n\rr\infty} \const^{n} P\big(M^{s}_{a n}>n\big) = 1 \} \leq -\frac{\log \const }{\log(1-p_{0})}.
\ee
Observe that the statement above can be equivalently written as
\be\label{M_n:upper}
\lim_{n\rr\infty} \const^{cn} P\big(M^{s}_{n}>cn\big) = 1, \q\mbox{for all } c\in\left(0,\frac{1}{\ol{a}}\right).
\ee

Next we consider part (c) of Theorem \ref{theorem-max-dis-upto}. We want to show
\be\label{a:lower_bd}
\ul{a}:=\sup\{a: \lim_{n\rr\infty} \const^{n} P\big(M^{s}_{a n}>n\big) = 0 \} >1,
\ee
which implies that
\[
\lim_{n\rr\infty} \const^{cn} P\big(M^{s}_{n}>cn\big) = 0, \q\mbox{for all } c>\frac{1}{\ul{a}}.
\]
To prove \eqref{a:lower_bd}, note that
\begin{equation}\label{Man_decomp}
\aligned
P\big(M^{s}_{a n}>n\big)
&= \sum_{n\leq j\leq a n}(1-p_{0})^{j}P(\tau_{n}=j)\\
&=\sum_{n\leq j\leq a n}m^{j}P(\tau_{n}=j).
\endaligned
\end{equation}
Since $\rho=(1+\sqrt{1-m^2})/m<2/m$, we have
\[
\rho^n m^{n}P(\tau_{n}=n) = \left(\rho m /2\right)^n\to 0,
\]
hence it suffices to show that there exists $a>1$ such that
\be\label{pgf_lower_bd_binary}
 \rho^n \sum_{n< j\leq a n} m^{j}P(\tau_{n}=j) \to 0.
\ee
Recall that $K(\theta)=E(\theta^{W_1})$. By the Chernoff bound we have for all $j$ and $\theta \geq 1$,
\[
P(\tau_{n}= j) \leq  P(W_{j}\geq n)
\leq \frac{E\big(\theta^{W_{j}}\big)}{\theta^n}
= \frac{(K(\theta))^j}{\theta^n}.
\]
Therefore,
\be\label{Man_lower_bd_binary}
\aligned
  \rho^n \sum_{n< j\leq a n} m^{j}P(\tau_{n}=j)
 &\leq \sum_{n<j\leq a n} \left(\frac{\rho}{\theta}\right)^n (m K(\theta))^j\\
 &=\sum_{n<j\leq a n} \exp\left(f(j/n,\theta)\cdot n\right),
 \endaligned
\ee
where
\[
f(x,\theta)=\log\rho -\log \theta + x\cdot \log(m K(\theta)),\q\mbox{for all } x> 1 \mbox{ and }\theta \geq 1.
\]
Plugging $K(\theta)=\big(\theta+\theta^{-1}\big)/2$ and differentiating $f(x,\theta)$ with respect to $\theta$ show that for any $x>1$, $f(x,\theta)$ attains its minimum at
\bn
\theta^{*}=\sqrt{\frac{x+1}{x-1}},
\en
with
\bn \label{dd}
f(x,\theta^{*})& =&\log \Big ( \frac{1+\sqrt{1-m^{2}}}{m}\Big)+ x\log (m x)-\frac{x-1}{2}\log(x-1)-\frac{x+1}{2}\log(x+1).
\en
Denote the function above by $g(x)$. It is easy to see that
\[
\lim_{x\to 1+} g(x) = \log \Big (1+\sqrt{1-m^{2}}\Big) - \log(2) <0.
\]
Moreover, by differentiating $g(x)$ we see that $g(x)$ is strictly increasing for $x<1/\sqrt{1-m^2}$ and strictly decreasing
for $x>1/\sqrt{1-m^2}$. It follows that there exists $1<\ul{a} \leq 1/\sqrt{1-m^2} \leq \ol{a}$ such that $g(x) <0$ for all $x<\ul{a}$, and by \eqref{Man_lower_bd_binary} we obtain~\eqref{pgf_lower_bd_binary}.

Finally we prove that $\ul{a}=\ol{a}=1/\sqrt{1-m^2}$, in other words, the function
\be\label{phi_jump}
\phi(a):= \lim_{n\rr\infty} \const^{n} P\big(M^{s}_{a n}>n\big)
 =\left\{\begin{array}{ll}
        0 & \q \mbox{if } a< 1/\sqrt{1-m^2}\\
        1 & \q \mbox{if } a> 1/\sqrt{1-m^2}.\\
        \end{array}
        \right.
\ee
Based on the analysis above, it is sufficient to show that
\[
g( 1/\sqrt{1-m^2})=0,\q\mbox{for all}\q m\in (0,1).
\]
To see this, denote $z=1/\sqrt{1-m^2}$. Then
\[\aligned
&g(z)\\
=&\log\Big(\frac{1+1/z}{m}\Big) + \log(mz) + (z-1)\log(mz) - \frac{z-1}{2}\log(z-1)-\frac{z+1}{2}\log(z+1)\\
=&\log(z+1)+ \frac{z-1}{2}\log\Big(\frac{m^2z^2}{z-1}\Big) -\frac{z+1}{2}\log(z+1)\\
=&\frac{z-1}{2}\log\Big(\frac{m^2z^2}{z^2-1}\Big),
\endaligned
\]
which equals zero by noting that $m^2z^2/(z^2-1)\equiv 1$ for $z=1/\sqrt{1-m^2}$.

The convergence in \eqref{phi_jump} implies that for any $\eps>0$,
\[\aligned
\rho^n P\big( M^{s}_{(1/\sqrt{1-m^2}- \eps) n}\leq n <M^{s}_{(1/\sqrt{1-m^2}+\eps) n}\big)
  \to 1.
\endaligned
\]
In fact, by using the Taylor expansion of $g(x)$ around $1/\sqrt{1-m^2}$ and noting that $g( 1/\sqrt{1-m^2})=g'( 1/\sqrt{1-m^2})=0$
while $g''( 1/\sqrt{1-m^2})<0$,
we can show that there exists  $c>0$ such that with $\eta_{\pm}(n):=
(1/\sqrt{1-m^2}  \pm c\sqrt{\log n/n})\cdot~n$,
\[
 \rho^n P\big( M^{s}_{\eta_{-}(n)}\leq n \leq M^{s}_{\eta_{+}(n)}\big)
 = \rho^n \sum_{j=\eta_{-}(n)}^{\eta_{+}(n)}  m^j P(\tau_n =j )\to 1.
\]
It follows that
\[
\limsup_{n\rr\infty} \rho^n m^{\eta_{+}(n)} P(\tau_n\in [\eta_{-}(n),\eta_{+}(n)])
\leq 1
\leq
\liminf_{n\rr\infty} \rho^n m^{\eta_{-}(n)} P(\tau_n\in [\eta_{-}(n),\eta_{+}(n)]).
\]
Taking logarithms and dividing by $n$ yield
\[
\lim_{n\rr\infty} \frac{\log P(\tau_n\in [\eta_{-}(n),\eta_{+}(n)])}{n} = -\left(\log\rho + 1/\sqrt{1-m^2}\log(m)\right).
\]
Noting  again that $\rho=(1+\sqrt{1-m^2})/m$, we can then rewrite the convergence above as follows:  for any $a>1$,
\be\label{first_passage_time_local_small_dev}
\aligned
\lim_{n\rr\infty}\frac{\log P(\tau_n \in [an \pm c\sqrt{n\, \log n}] )}{n}
&=-\log\left( (a+1)^{(a+1)/2} (a - 1)^{(a-1)/2} a^{-a}\right)\\
&:= -\log(\lambda(a)).
\endaligned
\ee
This can be regarded as a \emph{local small deviation} result for the first passage times~$\tau_n$ (recall that $\tau_n=O_p(n^2)$).
Loosely speaking, it says that $P(\tau_n \in [an \pm c\sqrt{n\,\log n}] )$ decays like $\lambda(a)^{-n}$. Observe that as $a\to 1$, $\lambda(a)\to 2$, so for $a$ close to 1,   $P(\tau_n \in [an \pm c\sqrt{n\,\log n}] )$ decays like $2^{-n}$, which is natural given the exact formula that $P(\tau_n = n)=2^{-n}$.  However, for general $a>1$, the convergence in \eqref{first_passage_time_local_small_dev} seems to be rather surprising.

\section{Existence of $\lim_{n\rr\infty } (-\log u(n)/n)$} \label{sec-proof-max-dis}
In this section we prove that $\lim_{n\rr\infty}(-\log u(n)/n)$ exists and is  positive. This result establishes the exponential decay of $u(n)$ as $n\rr\infty$ and is an important ingredient in the proof of Theorem \ref{theorem-max-dis}.

We start with the convergence in \eqref{lim-C}, which
may be of independent interest for
the study of small deviations of random walk.
In the following Proposition we state this result for a larger class of random walks which even does not require the
random walk to have a mean.

\begin{proposition}  \label{lemma-con-pgf}
Suppose that $(Y_i)_{i=1}^\infty$ is a sequence of \hbox{i.i.d.} (real) random variables. Define $S_n=\sum_{i=1}^n Y_i$ for all $n\in\zz{N}$, and for any $x\in \zz{R}$, let $\tau_{x}^{S}=\min\{n\geq 1:  S_{n}\geq x\}$. Then as long as $P(Y_1>0)>0$, for any $y>0$ and $\gamma\in(0,1)$ we have
\bq\label{eqn:pgf_tau_exp_decay}
f(y):=\lim_{x\rr\infty }\frac{\log E\big( \gamma^{ \tau^{S}_{x y}}\big)}{x}  \ \textrm{exists}.
\eq
Moreover, the limiting function $f(\cdot)$  satisfies that
\bq \label{q0}
f(y)= yf(1),
\eq
and
\bq \label{q00}
f(y)\geq y\log E\big(\gamma^{\tau^{S}_{1}}\big)>-\infty.
\eq
\end{proposition}
\begin{proof}
To ease the notation we write $\tau_{\cdot}$ instead of $\tau^{S}_{\cdot}$ throughout the proof.

For each $y>0$, define
\bn
a_{n}=\log E^{}\big( \gamma^{ \tau_{ny}}\big), \ n=1,2,...
\en
We first show that $\{-a_{n}\}_{n\geq 1}$ is a subadditive sequence. In fact, from the strong Markov property we get that for every $k,l>0$,
\bq \label{q1}
E\big(\gamma^{\tau_{(k+l)y}}\big)&=&E\big(\gamma^{\tau_{ky}}E^{S_{\tau_{ky}}}\big(\gamma^{\tau_{(k+l)y}-\tau_{ky}}\big)\big)  \nonumber \\
&\geq&E\big(\gamma^{\tau_{ky}}\big)E\big(\gamma^{\tau_{ly}}\big).
\eq
Hence
\[
\aligned
a_{k}+a_{l}
=&\log \big( E\big( \gamma^{ \tau_{k y}}\big)E\big( \gamma^{ \tau_{l y}}\big)\big)\\
\leq& \log \big( E\big( \gamma^{ \tau_{(k+l)y} }\big)\big)
= a_{k+l},
\endaligned
\]
in other words, $\{-a_{n}\}_{n\geq 1}$ is subadditive. By Fekete's subadditive lemma we then have that $f(y):=\lim_{n\rr\infty}\frac{\log E\big( \gamma^{ \tau_{n y}}\big)}{n}$ exists, moreover
\bq \label{q3}
f(y)=\sup_{n\geq 1}\frac{\log E\big( \gamma^{ \tau_{n y}}\big)}{n}.
\eq
The convergence along the whole sequence $\zz{R}\ni x\to\infty$ in \eqref{eqn:pgf_tau_exp_decay} follows from the monotonicity of $\tau_{y}$ in $y$.

Next we prove (\ref{q0}). Clearly $f(y)$ is decreasing in $y$, hence it suffices to show (\ref{q0}) for all rational numbers.
Let $y = i/j$ where $i,j \in \zz{N}$. From \eqref{eqn:pgf_tau_exp_decay} we get
\be \label{q4}
\aligned
f\Big(\frac{i}{j}\Big)
&=\lim_{n\rr\infty }\frac{\log E\big( \gamma^{ \tau_{ni/j }}\big)}{n}  \\
&=\lim_{n\rr\infty }\frac{\log E\big( \gamma^{ \tau_{nj(i/j) }}\big)}{nj}   \\
&=\frac{i}{j}\lim_{n\rr\infty }\frac{\log E\big( \gamma^{ \tau_{ni }}\big)}{ni}  \\
&=  \frac{i}{j}f(1),
\endaligned
\ee
and we get (\ref{q0}).

Finally, by (\ref{q0}), to prove (\ref{q00}) it suffices to show that
$f(1)\geq \log\big(E\big( \gamma^{\tau_{1}}\big)\big)>-\infty$. The first inequality follows from (\ref{q3}), and the second inequality holds due to that $P(Y_1>0)>0$ which implies  $P(\tau_1<\infty)>0$.
\end{proof}

\begin{remark}  \label{near-n-exmp}
If $S$ is a nearest neighbor random walk, then inequality (\ref{q1}) is an equality and we have
\bn
\frac{\log E\big(\gamma^{ \tau_{n y}}\big)}{n} = y\log E\big(\gamma^{\tau_{1}}\big), \ \textrm{for all }  n\in \zz{N} \mbox{ and }  y>0 \ \textrm{such that } ny\in \zz{N}.
\en
\end{remark}

Now we prove the statement in the section title.

\begin{lemma} \label{lemma-lim-w}
$\lim_{n\rr \infty} (-\log u(n)/n)$ \ \textrm{exists and} {belongs to} $ (0,\infty)$.
\end{lemma}
\begin{proof}
Again we will prove that $-\log u(n)$ is subadditive which implies the existence of the limit.
In fact, for any $k,l>0$, by the strong Markov property we have
\bn
u(k+l) &=& P\big(M \geq k+l\big)  \\
&\geq&P\big(M \geq k\big)P\big(M \geq k+l|M \geq k\big) \\
&\geq&P\big(M \geq k\big)P\big(M \geq l\big) \\
&=&u(k)u(l),
\en
and the subadditivity of $-\log u(n)$ follows. By Fekete's lemma again we obtain that
$$
\lim_{n\rr \infty} -\frac{\log u(n)}{n} = \inf_{n\geq 1}\left(-\frac{\log u(n)}{n}\right).
$$
In particular, we have
\bn
\lim_{n\rr \infty} -\frac{\log u(n)}{n}
\leq -\log u(1)
 =-\log P(M\geq 1)
 < \infty.
\en

Next we show that
$\lim_{n\rr\infty}(-\log u(n)/n)>0$.
Recall that $Z_{n}$ stands for the number of particles at generation $n$.
Note that
\be  \label{rn0}
\begin{aligned}
u(n) &= P(M\geq n) \\
&\leq P(Z_{n}>0) +P(M\geq n, Z_{n}=0)   \\
&\leq  m^{n}+ P(M_{n}\geq n).
\end{aligned}
\ee
Moreover,
 \be \label{rn1}
\begin{aligned}
P(M_{n}\geq n) &= \sum_{k=1}^{n}\sum_{i=0}^{\infty}P(Z_{k}=i,M_{k}\geq n, M_{k-1}<n) \\
&\leq  \sum_{k=1}^{n}\sum_{i=0}^{\infty}P(Z_{k}=i)\cdot i P(W_{k}\geq n) \\
&= \sum_{k=1}^{n}m^{k}P(W_{k}\geq n),
\end{aligned}
\ee
where in the second inequality we used the fact that the trajectory of each particle in generation $k$ follows the same law as the random walk $(W_j)_{j\leq k}$.

Recall the $K(\theta)$ was defined in (\ref{pgf}). For any $\theta_0\in (1,\rho(1/m))$,  by the monotonicity of $K(\cdot)$, we have $1<K(\theta_0)<1/m$. Moreover,  by Chernoff bound, 
\[
P(W_{k}\geq n)
\leq \frac{K(\theta_{0})^{k}}{\theta_{0}^{n}}.
\]
It follows from (\ref{rn1}) that
\begin{equation}\label{Mn_exp_tail}
P(M_{n}\geq n) \leq \sum_{k=1}^{n}m^{k} \frac{K(\theta_{0})^{k}}{\theta_{0}^{n}}  \leq \frac{1}{(1-m K(\theta_0)) \theta_{0}^{n}}, 
\end{equation}
which together with (\ref{rn0}) implies
\bn
u(n) \leq m^{n} +\frac{1}{(1-m K(\theta_0)) \theta_{0}^{n}}.
\en
It follows that $\lim_{n\rr\infty}(-\log u(n)/n)>0$.
\end{proof}

\section{A Discrete Feynman-Kac Formula } \label{sec-f-c}
In this section we derive a discrete Feynman-Kac formula for $u(n)$ which is one of the main ingredients in the proof of Theorem \ref{theorem-max-dis}. The derivation of the Feynman-Kac formula uses ideas from Section 2.2 in \cite{LS15}. We first introduce some additional definitions and a sequence of auxiliary lemmas.

Recall that the offspring distribution $F_{B} =\{p_k\}_{k\geq 0}$. Let $f(\cdot)$ be its probability generating function, and
define
\bq \label{2Q-funct1}
Q^{}(s)=1-\sum_{k=0}^{\infty}p^{}_{k}(1-s)^{k}, \  s\in [0,1],
\eq
which is related to $f(\cdot)$ via
\bq \label{f-Q}
f^{}(s)=1-Q^{}(1-s), \ \textrm{for all }  s\in[0,1].
\eq
Also recall that $F_{RW} = \{a_y\}_{y\in \zz{Z}}$. The following lemma gives a convolution equation for $ u(\cdot)$ based on $Q(\cdot)$.

\begin{lemma} \label{2lamma-v-1}
For all $  n\geq 1$,
\[
 u(n)=\sum_{y\in \zz{Z}}a_{y} Q\big(  u(n-y) \big).
\]
\end{lemma}
\begin{proof}
This is Proposition 5 in \cite{LS15} and is proved by conditioning on the first generation. More specifically,
by conditioning on the first generation, using the definition of our branching system and following the convention that $0^0=1$, we obtain
\bq \label{2rf11}
1- u(n) = \sum_{y\in \zz{Z}}a_{y}\sum_{k=0}^{\infty}  p_{k}(1-  u_{}(n-y))^{k}, \ \textrm{for all }  n\geq 1,
\eq
which is eqn.(10) in \cite{LS15} (there is a typo in eqn.(10) in \cite{LS15}. The summand $k$ should start from 0. The reason is that under the way that we define the branching system, if $Z_1=0$, then  no matter where the initial particle jumps to in the first generation, we always have $M=0$ as is explained in Remark 1 in  \cite{LS15}).
\end{proof}

Next, recall that $F_{B}$ has a finite third moment, hence by the Taylor expansion of~$Q(\cdot)$ at $s=0$ we have
\bq \label{Q-N}
 Q^{}(s) = ms -\frac{1}{2} \wt \sigma^2  s^{2}+O(s^{3}),
\eq
where $\wt  \sigma^2 =  \sigma^2 +m^{2}-m$.
Define
\bq \label{2h-N}
h^{}(s)=ms- Q(s) = \frac{1}{2}\wt \sigma^2  s^{2}+O(s^{3}),
\eq
and
\bq \label{2H-def1}
H^{}(s)=\frac{h^{}(s)}{m s}=\frac{1}{2}\frac{\wt \sigma^2 }{m} s+O(s^{2}).
\eq
Then Lemma \ref{2lamma-v-1} can be rewritten as the following  which is more useful for our purpose.
\begin{lemma}  \label{2lemma-v-id}
For all $  n \geq 1$,
\[
 u(n)=m\sum_{y \in \zz{Z}}a_{y} u(n-y)-\sum_{y \in \zz{Z}}a_{y}h\big( u(n-y)\big).
\]
\end{lemma}

We will also need the following result on the boundedness and monotonicity of $H$.
\begin{lemma}  \label{lemma-2H-bound}
For all $s\in [0,1]$,
\[
0\leq H^{}(s) \leq \frac{m-1+p_{0}}{m}.
\]
\end{lemma}
\begin{proof}
It is easy to verify  that $H(0)=0$ and
\bq \label{H-1}
H^{}(1)= \frac{m-Q(1)}{m}
=  \frac{m-1+p_{0}}{m}.
\eq
Note that $m\geq 1-p_{0}$, hence $H(1)\geq0$. To prove the desired conclusion, it is thus enough to show that $H$
is increasing.

By the definition of $H$ in (\ref{2H-def1}) and (\ref{f-Q})  we have
\bn
H(s)=1-\frac{1-f(1-s)}{ms}.
\en
Differentiating $H(s)$ and using the fact that $f(1)=1$ we get that $H$ is increasing if
\[
f'(1-s) \leq \frac{f(1)-f(1-s)}{s}, \q \textrm{for all }   s\in[0,1].
\]
This follows directly from the fact that $f$ is convex.
\end{proof}

Next, we denote by $\{ \W_{n}\}_{n \geq 0}$ a random walk on $\zz{Z}$ with the following law:
\bn
P(\W_{n+1}-\W_{n}=y\ |\ \W_{n},\W_{n-1},...)=a_{-y}, \q y \in \zz{Z},
\en
in other words, $\{\W_{n}\}_{n \geq 0}$ is a reflection of $W$, the random walk associated with our branching system.

Define the stopping times
\be  \label{tau-y-ex}
\bar \tau_{y} =\min\{k\geq 0: \W_{k}\leq y\}, \q\mbox{for all } y\geq 0, \q\mbox{and}\q \bar \tau:=\bar \tau_{0}.
\ee
Further define for each
$n\geq 0$,
\be \label{2mrt-z}
\aligned
Y^{}_{n}
=&m^{n} u(\W_{n})\mathds{1}_{\{\bar \tau\geq n\}} \prod_{j=1}^{n}\big[1-H^{}( u\big(\W_{j})\big)\big] \\
& +\sum_{i=1}^{n-1}m^{i-1}(1-p_{0})\mathds{1}_{\{\bar \tau=i \}}\prod_{j=1}^{i-1}\big[1-H^{}( u\big(\W_{j})\big)\big],
\endaligned
\ee
where we use the convention that
for any $k\leq 0,$ $\sum_{j=1}^{k}=0$ and $\prod_{j=1}^{k}=1$. In particular, $Y_0=u(\W_0)$.

Finally, let  $\mathcal{F}^{\W}=(\mathcal{F}^{\W}_{n})_{n\geq 0}$ be the natural filtration of $\{\W_{n}\}_{n \geq 0}$.

In the following lemma we prove that $Y^{}=\{Y_{n}\}_{n\geq 0}$ is a martingale.
\begin{lemma}  \label{2lemma-mart-z}
If $\W_0=x\geq 0$, then $Y^{}$ is a martingale with respect to $\mathcal F^{\W}$.
\end{lemma}
\begin{proof}
Define
\[\aligned
Y^{1}_{n}&:= m^{n} u(\W_{n})\mathds{1}_{\{\bar \tau\geq n\}} \prod_{j=1}^{n}\big[1-H^{}( u\big(\W_{j})\big)\big], \q \mbox{and}\\
Y^{2}_{n}&:= \sum_{i=1}^{n-1}m^{i-1}(1-p_{0})\mathds{1}_{\{\bar \tau=i \}}\prod_{j=1}^{i-1}\big[1-H^{}( u\big(\W_{j})\big)\big].
\endaligned
\]
Note that $Y^{2}_{n+1} \in \mathcal{F}_{n}^{\W}$, therefore $Y^{}$ is a martingale iff
\bq \label{c1}
E\big(Y^{1}_{n+1}|\mathcal F_{n}^{\W}\big) &=& Y^{1}_{n}+Y^{2}_{n}-Y^{2}_{n+1} \nonumber \\
&=&Y^{1}_{n}-m^{n-1}(1-p_{0})\mathds{1}_{\{\bar \tau=n \}}\prod_{j=1}^{n-1}\big[1-H^{}( u\big(\W_{j})\big)\big].
\eq
We distinguish among the following cases:

\noindent\textbf{Case 1:} $\bar \tau<n$, then $Y^{1}_{n+1} =Y^{1}_{n}= 0$, and (\ref{c1}) holds trivially.

\noindent\textbf{Case 2:}   $\bar \tau=n$, then $Y^{1}_{n+1} = 0$. From (\ref{H-1}) we get
$1-H(1)=(1-p_{0})/m.$ Hence, noting that $\bar \tau=n$ so that $u(\W_n)=1$ we get
\bn
 Y^{1}_{n} &=&m^{n} u(\W_{n}) \prod_{j=1}^{n}\big[1-H^{}( u\big(\W_{j})\big)\big]\nonumber  \\
&=&m^{n} \cdot1  \cdot \big(1-H(1)\big) \prod_{j=1}^{n-1}\big[1-H^{}( u\big(\W_{j})\big)\big]\nonumber  \\
&=& m^{n-1}(1-p_{0})\prod_{j=1}^{n-1}\big[1-H^{}( u\big(\W_{j})\big)\big],
\en
and (\ref{c1}) follows.

\noindent\textbf{Case 3:}  $\bar \tau\geq n+1$, then $\W_{n}\geq 1$. From the definition of $H$ in (\ref{2H-def1}) and Lemma \ref{2lemma-v-id} we obtain that
\be \label{I}
\aligned
&E\big(Y^{1}_{n+1}|\mathcal F^{\W}_{n}\big) \\
= &m^{n+1}\prod_{j=1}^{n}\big[1-H^{}\big( u(\W_{j})\big)\big]\cdot E\bigg( u(\W_{n+1})\big[1-H^{}\big( u(\W_{n+1})\big)\big] \Big | \mathcal F^{\W}_{n} \bigg)   \\
=&  m^{n+1}\prod_{j=1}^{n}\big[1-H\big( u(\W_{j})\big)\big] \cdot E\bigg( u(\W_{n+1})-\frac{h\big( u(\W_{n+1})\big) }{m} \Big | \mathcal F^{\W}_{n} \bigg)   \\
=& m^{n+1}\prod_{j=1}^{n}\big[1-H^{}\big( u(\W_{j})\big)\big]\cdot \sum_{y\in\zz{Z}}a_{y}\bigg( u(\W_{n}-y)-\frac{h^{}\big( u(\W_{n}-y)\big)}{m} \bigg)  \\
=& m^{n}u({\W}_{n})\prod_{j=1}^{n}\big[1-H\big( u(\W_{j})\big)\big].
\endaligned
\ee
Identity (\ref{c1}) again follows.
\end{proof}

To avoid additional notation, we will also use
$P^{x}, \ E^{x}$
to denote the probability and expectation under the distribution  of $\{\W_{n}\}_{n \geq 0}$ with $\W_{0}=x$, and omit the superscript when $x=0$ (and when there is no confusion).

Finally we  are ready to derive a discrete Feynman-Kac formula for $u(\cdot)$.
\begin{lemma} \label{f-k}
For all $y \geq 0$ and $x\geq y$ we have
\bq \label{rf220}
 u(x)=E^{x}\Big(m^{\bar \tau_{y}} u(\W_{\bar \tau_{y}})\prod_{j=1}^{\bar \tau_{y}  }\big[1-H^{}( u\big(\W_{j})\big)\big] \Big).
\eq
\end{lemma}
\begin{proof}
Clearly $\bar\tau_{y}\leq \bar \tau$ for every $y\geq0$. Since $m\leq 1$, $Y$ is a bounded martingale, so by the Optional Stopping Theorem we have
\bn
u(x)=Y_{0}=E^{x}\big(Y_{\bar\tau_{y}}\big).
\en
Using again that $\bar \tau_{y}\leq \bar\tau$ it is easy to verify that
\bn
Y_{\bar \tau_{y}} &=&m^{\bar \tau_{y}} u(\W_{ \bar\tau_{y}})\prod_{j=1}^{\bar \tau_{y}}\big[1-H^{}( u\big(\W_{j})\big)\big].
\en
\end{proof}

\section{Proof of Theorem \ref{theorem-max-dis}} \label{sec-proof-theorem-max-dis}

We start with some notation and a few auxiliary lemmas. Recall that $\tau_{y}$ was defined in (\ref{tau-y}).
\begin{lemma}   \label{Lemma-PG-V}
$$
E\big(\theta^{W_{\tau_{1}}}\big) < \infty, \q \textrm{for all } \theta \geq 1.
$$
\end{lemma}
\begin{proof}
For any $k\geq 1$, we have
\be \label{Sp1}
P(W_{\tau_{1}}=k) = \sum_{i =0}^{\infty} G_{-}(0,-i)\cdot p_{k+i},
\ee
where
\bn
G_{-}(0,-i) = \sum_{n=0}^{\infty}P(W_{n}=-i, \ n <\tau_{1}),
\en
denotes the Green's function inside the half line $(-\infty,0]$. By proposition 18.8 (p. 203) and Proposition 19.3 (p. 209) in \cite{spitzer}, there exists $A>0$ such that $G_{-}(0,-i)<A$ for all $i\geq 0$ (see also the proof of Proposition 19.4 in \cite{spitzer}). It follows from (\ref{Sp1}) that
\be \label{P-V}
P(W_{\tau_{1}}=k) \leq A\sum_{i=0}^{\infty}p_{k+i} =AP(W_{1}\geq k), \q \textrm{for all } k\geq0.
\ee
The conclusion then follows from \eqref{pgf-cond}.
\end{proof}

Denote $\xi_{n}(c)= e^{-e^{-c n}}m$ for $n\geq 1$ and $c>0$. Also define for any $s\in(0,1)$ and $k=0,1,...$,
\bq \label{p-k}
 p_{k}(s)= E\big(s^{ \tau_{1}}\indic_{\{W_{ \tau_{1}}=1+k\}}\big).
 \eq
{Recall that $\ell(n)= \const\big(m^{-1}\big)^{n}u(n)$.}  The following lemma gives recursive bounds on~$\ell(n)$.
\begin{lemma} \label{lemma-lineq}
For all $n\geq 0$,
\begin{equation}\label{ln_ma_upper}
\ell(n+1) \leq \sum_{k=0}^{{\infty}}  \const\big(m^{-1}\big)^{k+1} p_{k}(m)\cdot\ell(n-k).
\end{equation}
Moreover, there exist $c>0$ and $N_{0}>0$ such that for all $n>N_{0}$,
\be\label{ln_ma_lower}
\ell(n+1)\geq \sum_{k=0}^{\lfloor n /2 \rfloor}  \const\big(m^{-1}\big)^{k+1} p_{k}(\xi_{n}(c))\cdot\ell(n-k).
\ee
Furthermore, $w_{k}:= \const\big(m^{-1}\big)^{k+1} p_{k}(m)$ satisfy that
\be\label{sum-w}
 \sum_{k=0}^{{\infty}}  w_k = 1.
\ee
\end{lemma}
\begin{proof}
From Lemmas \ref{f-k} and \ref{lemma-2H-bound} we get that for every $n\geq 0$,
\bq \label{r2-2}
 u(n+1)&=&E^{n+1}\Big(m^{\bar \tau_{n}} u(\W_{\bar \tau_{n}})\prod_{j=1}^{\bar \tau_{n}  }\big[1-H^{}( u\big(\W_{j})\big)\big] \Big)\\
 &\leq & E^{n+1}\Big(m^{\bar \tau_{n}} u(\W_{\bar \tau_{n}}) \Big).  \nonumber
\eq
Decomposing  $E^{n+1}\Big(m^{\bar \tau_{n}} u(\W_{\bar \tau_{n}}) \Big)$ as
\[
  \sum_{k=0}^{{\infty}} E^{n+1}\Big(m^{\bar \tau_{n}}\indic_{\{\W_{\bar \tau_{n}}=n-k\}}\Big)u(n-k)
  =  \sum_{k=0}^{{\infty}} E\Big(m^{\tau_{1}}\indic_{\{W_{\tau_{1}}=1+k\}}\Big)u(n-k),
\]
we  see that \eqref{ln_ma_upper} holds.

To prove the lower bound \eqref{ln_ma_lower}, note that there exists $C_1>0$ such that
\bq \label{aa1}
\log(1-x)\geq - C_1 x, \ \textrm{for all }  x\in[0,1/2].
\eq
Moreover, by our assumptions $F_B$ has a finite third moment, hence there exists $C_2>0$ such that error term in (\ref{2H-def1}) is bounded by $C_2 s^2$ for all $s\in[0,1]$, and so there exists $C_3>0$ such that
\bq \label{aa2}
H(s) \leq C_3 s, \ \textrm{for all }  s\in[0,1].
\eq
Furthermore, by the monotonicity of $u(\cdot)$ we get that
for all $n$ large enough, if $\W_{\bar\tau_{n}}> \lfloor n /2\rfloor$, then
\bq \label{aa3}
u\big(\W_{j}\big) \leq u(\lfloor n /2\rfloor )
\leq  1/(2C_3),
\ \textrm{for all }  j\leq \bar \tau_{n},
\eq
and therefore by (\ref{aa1})--(\ref{aa3}),
\be \label{hg111}
\aligned
\prod_{j=1}^{\bar \tau_{n}  }\big[1-H^{}( u\big(\W_{j})\big)\big] &\geq
\prod_{j=1}^{\bar \tau_{n} }\big[1- C_3\ u\big(\W_{j})\big] \\
&\geq \exp\Big\{-C\sum_{j=1}^{\bar \tau_{n}}  u\big(\W_{j})\Big\},
\endaligned
\ee
where $C=C_1 C_3$. It follows from (\ref{r2-2}) and (\ref{hg111})   that
\[
\aligned
u(n+1)
\geq & E^{n+1}\Big(m^{\bar\tau_{n}}u(\W_{\bar \tau_{n}}) \exp \Big \{-  C \sum_{j=1}^{\bar\tau_{n}} u(\W_{j}) \Big\}\ {\mathds{1}_{\{\W_{\bar\tau_{n}}> \lfloor n /2\rfloor }\}}\Big).
\endaligned
\]
Moreover, by Lemma \ref{lemma-lim-w},  there exists $\bar c>0$ such that for all $n$ large enough,
\[
u(n/2) \leq e^{-\bar cn}.
\]
Therefore, there exists $c>0$ such that for all $n$ large enough,
 \bq
\ell(n+1)
&\geq & \const \big(m^{-1}\big)^{n+1} E^{n+1}\Big(m^{\bar\tau_{n}}u(\W_{\bar \tau_{n}}) \exp \Big \{- C \sum_{j=1}^{\bar\tau_{n}} u(\W_{j}) \Big\} {\mathds{1}_{\{\W_{\bar\tau_{n}}> \lfloor n /2 \rfloor}\}}\Big)   \nonumber \\
&\geq & \const \big(m^{-1}\big)^{n+1} \sum_{k=0}^{\lfloor n/2 \rfloor }E^{n+1}\Big(m^{\bar\tau_{n}}u(\W_{\bar \tau_{n}}) \exp \Big \{- C \sum_{j=1}^{\bar\tau_{n}} u(\W_{j}) \Big\}\indic_{\{\mathcal W_{\bar \tau_{n}}=n-k\}}\Big)  \nonumber  \\
&\geq & \const \big(m^{-1}\big)^{n+1} \sum_{k=0}^{\lfloor n /2 \rfloor}E^{n+1}\Big(m^{\bar\tau_{n}}u(\W_{\bar \tau_{n}}) \exp \Big \{ - C u(\lfloor n /2 \rfloor)\bar\tau_{n} \Big\}\indic_{\{\mathcal W_{\bar \tau_{n}}=n-k\}}\Big)  \nonumber  \\
&\geq &\sum_{k=0}^{\lfloor n /2 \rfloor}  \const \big(m^{-1}\big)^{n+1} E^{n+1}\Big((\xi_{n}(c))^{\bar \tau_{n}}\indic_{\{\mathcal W_{\bar \tau_{n}}=n-k\}}\Big)u(n-k) \nonumber \\
&=& \sum_{k=0}^{\lfloor n /2 \rfloor}  \const \big(m^{-1}\big)^{k+1} E\Big((\xi_{n}(c))^{ \tau_{1}}\indic_{\{W_{ \tau_{1}}=1+k\}}\Big)\ell(n-k) \nonumber \\
&=&\sum_{k=0}^{\lfloor n /2 \rfloor}  \const \big(m^{-1}\big)^{k+1} p_{k}(\xi_{n}(c))\cdot \ell(n-k). \nonumber
 \eq

Finally, the proof of (\ref{sum-w}) uses ideas from Proposition 3.1(i) in \cite{Vidmar14}. By the definition of $\rho(m^{-1})$, $\big\{m^n\const(m^{-1})^{W_{n}}\big\}_{n\geq 0}$ is a martingale. Hence, for any $N\in \mathbb N$, by the  Optional Stopping Theorem, we have $E\big(m^{\tau_{1}\wedge N}\const(m^{-1})^{W_{\tau_{1}\wedge N}}\big)=1$.  Lemma \ref{Lemma-PG-V} and the Dominated Convergence Theorem then imply that $E\big(m^{\tau_{1}}\const(m^{-1})^{W_{\tau_{1}}}\big)=1$, which is (\ref{sum-w}).
\end{proof}

In the rest of this section, we assume that $c$ is fixed as in Lemma \ref{lemma-lineq}, and we suppress the dependence of $\xi_{n}=\xi_{n}(c)$ in $c$.
In the following lemma we prove that $\ell(\cdot)$ is bounded from both above and below.
\begin{lemma} \label{lemma-ln-bnd}
$
0<\liminf_{n\rr\infty }\ell(n)\leq  \sup_{n}\ell(n) \leq 1.
$
\end{lemma}
\begin{proof}
We first prove the upper bound on $\ell(n)$. {Denote
\bn
T_{n}=\sup_{k\leq n} \ell(k).
\en
Note that $\ell(n) \leq 1$ for $n\leq 0$, hence $T(n)<\infty$  for all $n$.}
By \eqref{ln_ma_upper} and \eqref{sum-w} we have
\bn
\ell(n+1) \leq  T_{n}\sum_{k=0}^{\infty} w_k
 = T_{n}.
\en
Hence $T_{n+1}\leq T_{n}$ and immediately follows $\sup_{n\geq 1}\ell(n)\leq T_0 = 1$.

Next we prove the lower bound. Let $\const =\const(m^{-1})$, and
\bn
J_{n}=\min\{\ell(n),\ell(n-1)...,\ell(\lfloor n/2\rfloor )\}.
\en
Firstly, note that
\be \label{rf1}
\aligned
\sum_{k=\lfloor n/2 \rfloor}^{\infty}w_{k} &=\sum_{k= \lfloor n/2 \rfloor }^{\infty} \const^{k+1} E\big(m^{ \tau_{1}}\indic_{\{W_{ \tau_{1}}=1+k\}}\big) \\
&\leq E\big(\const^{W_{\tau_{1}}}\indic_{\{W_{\tau_1}>\lfloor n/2 \rfloor\}}\big).
\endaligned
\ee
By the Cauchy--Schwarz inequality  and Lemma \ref{Lemma-PG-V}, we get that there exists $C>0$ such that for every $k\geq 1$,
\be
\aligned
E\big(\const^{W_{ \tau_{1}}}\indic_{\{W_{ \tau_{1}}>k\}}\big)
&\leq  \sqrt{E\big(\const^{2W_{ \tau_{1}}}\big)\cdot P\big(W_{\tau_{1}}>k\big)} \\
&\leq \sqrt{\frac{\left(E\big(\const^{2W_{ \tau_{1}}}\big)\right)^2}{\rho^{2k}}}\\
&\leq C\rho^{-k}.
\endaligned
\ee
Furthermore, use the bound $1-e^{-x} \leq x$ for $x>0$ to get
\be  \label{rf3}
\aligned
p_{k}(m)-p_{k}(\xi_{n})
&=E\big(m^{ \tau_{1}}(1-e^{-e^{-cn}\tau_1})\indic_{\{W_{ \tau_{1}}=1+k\}}\big) \\
&\leq e^{-cn}E\big(\tau_1 m^{ \tau_{1}}\indic_{\{W_{ \tau_{1}}=1+k\}}\big).
\endaligned
\ee
Therefore, by \eqref{ln_ma_lower}, \eqref{sum-w},  and (\ref{rf1})--(\ref{rf3}), for all $n$ large enough,
 \be \label{ln_inductive_lower_bd}
 \aligned
\ell(n+1)
&\geq \sum_{k=0}^{\lfloor n/2 \rfloor}  \const^{k+1} \left(p_{k}(m)-e^{-cn}E\big(\tau_1 m^{ \tau_{1}}\indic_{\{W_{ \tau_{1}}=1+k\}} \big)\right)\cdot \ell(n-k)  \\
&\geq J_{n}\left(\sum_{k=0}^{\lfloor n/2 \rfloor}w_{k} -e^{-cn}E \Big(\tau_1  m^{ \tau_{1}} \sum_{k=0}^{\infty} \rho^{k+1}\indic_{\{W_{ \tau_{1}}=1+k\}} \Big) \right)  \\
&\geq J_{n}\big(1- C\rho^{- n/2} - C' e^{-c n}\big),
\endaligned
\ee
where in the last inequality we used the fact that $\sup_{k\geq 0} k m^k <\infty$ so that
\[
 E \Big(\tau_1  m^{ \tau_{1}} \sum_{k=0}^{\infty} \rho^{k+1}\indic_{\{W_{ \tau_{1}}=1+k\}} \Big)
 \leq C_1  E \Big(\sum_{k=0}^{\infty} \rho^{k+1}\indic_{\{W_{ \tau_{1}}=1+k\}} \Big)
 = C_1 E(\rho^{W_{\tau_1}}),
\]
which is finite by Lemma \ref{Lemma-PG-V}.
It follows that
\[
J_{n+1}\geq J_{n}\big(1- C \rho^{- n/2} - C' e^{-c n}\big).
\]
Since
$
\prod_{n\geq 1}\big(1- C \rho^{- n/2} - C' e^{-c n} \big)>0,
$
we get that
$
\liminf_{n\rr\infty }J_{n} >0,
$
which implies that
$
\liminf_{n\rr\infty }\ell_{n} >0.
$
\end{proof}

 Before we give the next lemma we introduce some additional definitions.
 Define
 \bq \label{al-def}
 \kappa = \limsup_{n\rr\infty }\ell(n)\in (0,1].
 \eq
Let $\{n_{k}\}_{k\geq 0}$ be a subsequence which satisfies
\bn
\lim_{k\rr\infty }\ell(n_{k})=\kappa.
\en
Recall that
$
w_{k}= \const\big(m^{-1}\big)^{k+1} p_{k}(m)
$
satisfy $\sum_{k=0}^{\infty}w_{k}=1$ by  \eqref{sum-w}.

\begin{lemma} \label{lemma-lim}
 If $W$ {has a finite right range $R$} and is nearly right-continuous, then
for every $\eps>0$, there exists $N_{0}>0$ such that if $n_{k}>N_{0}$, then
 \begin{itemize}
 \item[\bf{(a)}]
$
\ell(n_{k}-i)>\kappa - \eps, \  \textrm{for every } i=0,...,R-1.
$
 \item[\bf{(b)}] There exists $C>0$ such that
\be\label{ln_lower_bd}
\ell(n_{k}+i)>\kappa - \eps -C\sum_{r=0}^{i-1}e^{-c (n_{k}+r)}, \  \textrm{for every } i=0,1,...
\ee
\item[\bf{(c)}] Consequently, $\lim_{n\rr\infty }\ell(n)=\kappa.$
\end{itemize}
 \end{lemma}
 \begin{proof}
(a) Define
$$
\underline w =\min_{k=0,...,R-1}w_{k}\ \in (0,1).
$$
{Note that $\underline w  >0$ since $W$ has a finite right range and is nearly right-continuous. }

Fix $N_{0}$ large enough such that
\bq \label{ell-ineqee}
\ell(n) \leq \kappa+\eps \underline w /2, \ \textrm{ for all } n>N_{0}-R.
\eq
and
\bq \label{ell-ineq}
\ell(n_{k}) > \kappa-\eps \underline w /2, \ \textrm{ for all } n_{k}>N_{0}.
\eq
Since $W$ has a finite right range $R$,
\be \label{p-k-finite}
p_{k}(m)=0, \q \textrm{ for all } k\geq R.
\ee
Hence, by  (\ref{ln_ma_upper}), (\ref{sum-w}), (\ref{ell-ineqee}) and \eqref{ell-ineq}  we get for every $n_{k}>N_{0}$,
\[
\aligned
\kappa-\eps \underline w /2
&<\ell(n_{k})\leq \sum_{j=0}^{R-1}  w_{j}\ell(n_{k}-j-1)   \\
&\leq  w_{i}\ell(n_{k}-i-1)+ (1-w_{i})(\kappa+\eps \underline w /2), \ \textrm{ for every } i=0,...,R-1.
\endaligned
\]
This immediately implies (a).

\smallskip
(b) We will prove (b) by induction.

By (a),  \eqref{ln_lower_bd} holds for $i=0$.
Now suppose that \eqref{ln_lower_bd} is satisfied for $i$. Observe that \eqref{ln_inductive_lower_bd}, (\ref{p-k-finite})
 and the fact that $\sup_n \ell(n)\leq 1$ proven in Lemma \ref{lemma-ln-bnd} imply that there exists $C>0$ such that for all~$n$ large enough,
\[
\ell(n+1)
\geq \sum_{j=0}^{R-1} w_j {\ell(n-j)} -  Ce^{-c n}.
\]
Hence by the induction hypothesis we obtain that
\bn
 \ell(n_{k}+i+1)
 &\geq&  \sum_{j=0}^{R-1}  w_{j}\ell(n_{k}+i-j) -Ce^{-c (n_{k}+i)} \\
 &\geq&  \sum_{j=0}^{R-1}  w_{j}\Big(\kappa-\eps-C\sum_{r=0}^{i-j-1}e^{-c (n_{k}+r)}\Big) -Ce^{-c (n_{k}+i)} \\
 &\geq&  \sum_{j=0}^{R-1}  w_{j}\Big(\kappa-\eps-C\sum_{r=0}^{i-1}e^{-c (n_{k}+r)}\Big) -Ce^{-c (n_{k}+i)} \\
  &=&  \kappa-\eps  -C\sum_{r=0}^{i}e^{-c (n_{k}+i)},
\en
i.e., \eqref{ln_lower_bd} holds for $i+1$.

\smallskip
(c) By (\ref{al-def}) and Part (b) we get that for every $\eps>0$, there exist $C>0$ and $N_{0}>0$ such that for every $n_{k}>N_{0}$,
 \bn
\kappa +\eps > \ell(n_{k}+i)>\kappa - \eps -Ce^{-c n_{k}}, \  \textrm{for all } i=0,1,...,
\en
and the conclusion follows.
 \end{proof}

Now we are ready to prove Theorem \ref{theorem-max-dis}.
\begin{proof}[Proof of Theorem \ref{theorem-max-dis}]
(a) and (b) have been proved in Lemmas \ref{lemma-ln-bnd} and \ref{lemma-lim}(c) respectively.

(c) Define for every $n\geq1$,
 \bq \label{l-bar}
 \bar \ell(n)=\const(m^{-1})^{n} E\big(m^{\tau_{n}}\big).
 \eq
Recall that $p_{k}(m)$ was defined in (\ref{p-k}). From the strong Markov property we have
 \bn
 E\big(m^{\tau_{n+1}}\big) &=&\sum_{k=0}^{\infty}E\big(m^{\tau_{1}}\mathds{1}_{\{W_{\tau_{1}}=k+1\}}m^{\tau_{n+1}-\tau_{1}}\big) \\
 &=&\sum_{k=0}^{\infty}E\big(m^{\tau_{1}}\mathds{1}_{\{W_{\tau_{1}}=k+1\}}\big)E\big(m^{\tau_{n-k}}\big)\\
&=&\sum_{k=0}^{\infty}p_{k}(m)E\big(m^{\tau_{n-k}}\big).
  \en
Therefore for every $n\geq 1$,
 \bq \label{l-bar}
  \bar \ell(n+1)= \sum_{k=0}^{\infty} \const(m^{-1})^{k+1}p_{k}(m)\ \bar \ell(n-k).
 \eq
 Just as in the proof of Lemma \ref{lemma-ln-bnd}, we can use (\ref{l-bar}) to show that $  \bar \ell(n)$ is bounded from both below and above. In fact, the proof in this case is simpler since we have an exact recursion equation for $\bar \ell(n)$ instead of bounds as in Lemma \ref{lemma-lineq}. It follows that
 \bn
\log \const(m^{-1})= - \lim_{n\rr\infty}\frac{\log E\big(m^{\tau_{n}}\big)}{n}.
 \en
\end{proof}

\section{Proof of Theorem \ref{theorem-max-dis-upto}}  \label{Sec-pf-upto}
We first derive a lower bound of $M_n$ in terms of $M$.
Let $\const =\const(m^{-1})$.
\begin{lemma}  \label{lemma-1-bnd}
For every $c\in\big(0,-\frac{\log m}{\log \const}\big)$, we have
\[
\lim_{n\to\infty} P\big( M_{n} \geq c n \ |\ M\geq cn\big) =1.
\]
\end{lemma}
\begin{proof}
We have
\bn
P\big(M\geq cn\big)
= P\big( M\geq cn,\ M_{n}< cn \big)+ P\big( M_{n}\geq cn \big),
\en
hence
\[
P\big( M_{n} \geq c n \ |\ M\geq cn\big) =1 - \frac{P\big( M\geq cn,\ M_{n}< cn \big)}{P\big(M\geq cn\big) }.
\]
Observe that
\[
  P\big( M\geq cn,\ M_{n}< cn \big)\leq P(Z_n\geq 1)\leq E(Z_n) = m^n.
\]
By our choice of $c$ and Theorem \ref{theorem-max-dis} we then get that
\[
\frac{P\big( M\geq cn,\ M_{n}< cn \big)}{P\big(M\geq cn\big) }\to 0,
\]
and the conclusion follows.
\end{proof}

Lemma \ref{lemma-1-bnd} can be reformulated in the following more informative way, which indicates that the main contribution to the event $\{M\geq n\}$ is $\{M_{an}>n\}$ for some $a>0.$
\begin{cor}\label{cor:M_Mn}
For every $a>-\frac{\log \const}{\log m}$, we have
\[
\lim_{n\to\infty} P\big( M_{an} \geq  n \ |\ M\geq n\big) =1.
\]
\end{cor}

We are now ready to prove Theorem \ref{theorem-max-dis-upto}.
\begin{proof}[Proof of Theorem \ref{theorem-max-dis-upto}]
Parts (a) and (b) are direct consequences of Lemma \ref{lemma-1-bnd} and Theorem \ref{theorem-max-dis}.

Now we prove (c). We first need to show that there exists $c<\infty$ such that $\rho^{cn} P(M_n\geq cn)\to 0$. This is equivalent to that when $a>0$ is small enough,
\begin{equation}\label{Man_exp_tail}
 \rho^{n} P(M_{an} \geq n) \to 0.
\end{equation}
By a similar argument to that for \eqref{Mn_exp_tail}, we can show that
\[
P(M_{an} \geq n) \leq \sum_{k\leq an} m^k \frac{(K(\rho + 1))^{an}}{(\rho + 1)^n}\leq \frac{(K(\rho + 1))^{an}}{(1-m)(\rho + 1)^n}.
\]
The convergence \eqref{Man_exp_tail} follows by taking $a>0$ small enough so that $\rho\cdot (K(\rho + 1))^{a}/(\rho + 1) < 1$.

Next we show that if the random walk $W$ has a finite right range $R$, then $\ol{\delta} < R$. This
is equivalent to that there exists $a> 1/R$ such that
\eqref{Man_exp_tail} holds.
We first extend \eqref{Man_decomp} to current general setting when there can be multiple particles at each generation.
Similar to \eqref{rn1} and using the finite right range assumption, we have
\[
P\big(M_{a n}\geq n\big)
\leq \sum_{n/R\leq j\leq a n} m^j P(W_j \geq n).
\]

Next, since $\rho > 1$ satisfies that
\[
K(\rho)=\sum_{y\leq R} \rho^y a_y = \frac{1}{m},
\]
we have
\be\label{rho_m_aR}
  m \rho^R a_R < 1,
\ee
and hence
\[
\rho^n m^{n/R}P(W_{n/R}=n) \leq \left(\rho^R m a_R\right)^{n/R}\to 0.
\]
(The ``<'' sign is included to cover the case when $n/R\not\in\zz{N}$ in which case\\ $P(W_{n/R}=n)=0$.)
Therefore to show \eqref{Man_exp_tail}, it is sufficient to prove that there exists $a>1/R$ such that
\be\label{pgf_lower_bd}
 \rho^n \sum_{n/R< j\leq a n} m^{j}P(W_j \geq n) \to 0.
\ee
Exactly as in \eqref{Man_lower_bd_binary} we have
\[
  \rho^n \sum_{n/R< j\leq a n} m^{j}P(W_j \geq n)
 \leq \sum_{n/R<j\leq a n} \exp\left(f(j/n,\theta)\ n\right),
\]
where
\[
f(x,\theta):=\log\rho -\log \theta + x\cdot \log(m K(\theta)),\q\mbox{for all } x> 1/R \mbox{ and }\theta \geq 1.
\]
For any fixed $x> 1/R$, differentiating $f$ with respect to $\theta$ shows that $f(x,\cdot)$ attains its minimum at $\theta^*=\theta^*(x)$ which is the unique solution to
\be\label{eq_theta}
\frac{1}{x} = \frac{\theta K'(\theta)}{K(\theta)}.
\ee
The solution exists and is unique since if we denote the function on the right hand side by $h(\theta)$, then
\begin{enumerate}
\item[(i)] $h(1)=0$ and $\lim_{\theta\to\infty} h(\theta)=R$; and
\item[(ii)]
\[
h'(\theta)= \frac{\sum_{y< z\leq R} (y^2 + z^2 - 2yz ) \theta^{y+z-1} a_y a_z}{(K(\theta))^2}>0.
\]
\end{enumerate}
Denote
\[\aligned
 g(x) &= \log\rho -\log \theta^* + x\cdot \log(m K(\theta^*))\\
      &=\log\rho + x\log m+ \log x + \log(K'(\theta^*)) + (x-1)\log(K(\theta^*)).
 \endaligned
\]
It is then sufficient to show that $g(x)$ is negative as $x\to 1/R +$. This is true because as $x\to 1/R+$, we have $1/x\to R-$, and according to equation \eqref{eq_theta}, we must have that $\theta^*\to\infty$. It is then easy to show that
\[
 \lim_{x\to 1/R+} (\log(K'(\theta^*)) + (x-1)\log(K(\theta^*))
 =\log R + \frac{\log (a_R)}{R}.
\]
Further note
\[
\lim_{x\to 1/R+} (\log\rho + x\log m+ \log x)=\log(\rho m^{1/R}) -\log R.
\]
Combining the two limits above with \eqref{rho_m_aR} we see that indeed  $g(x)$ is negative as $x\to 1/R +$, and \eqref{pgf_lower_bd} follows.
\end{proof}

\section*{Acknowledgements}
We are very grateful to an anonymous referee for  careful reading of the manuscript,
and for a number of useful comments and suggestions that significantly improved this paper.

\bigskip
\noindent Eyal Neuman: Department of Mathematics, University of Rochester, Rochester, 14627 NY, USA. eneuman4@ur.rochester.edu

\medskip
\noindent Xinghua Zheng: Department of Information Systems
 Business Statistics and Operations Management, Hong Kong University of Science and
Technology, Clear Water Bay, Kowloon, Hong Kong. xhzheng@ust.hk


\begin{thebibliography}{10}

\bibitem{Aidekon2013}
E.~Aidekon.
\newblock Convergence in law of the minimum of a branching random walk.
\newblock {\em Ann. Probab.}, 41(3A):1362--1426, 2013.

\bibitem{Athreya-1972}
K.~B. Athreya and P.~E. Ney.
\newblock {\em Branching processes}.
\newblock Springer, New York, 1972.

\bibitem{Bachmann:2000}
M.~Bachmann.
\newblock Limit theorems for the minimal position in a branching random walk
  with independent logconcave displacements.
\newblock {\em Advances in Applied Probability}, 32(1):159--176, 2000.

\bibitem{Biggins}
J.~D. Biggins.
\newblock The first- and last-birth problems for a multitype age-dependent
  branching process.
\newblock {\em Advances in Applied Probability}, 8(3):446--459, 1976.

\bibitem{Bramson-Ding-Z:2014}
M.~Bramson, J.~Ding, and O.~Zeitouni.
\newblock Convergence in law of the maximum of nonlattice branching random
  walk.
\newblock {\em arXiv preprint arXiv:1404.3423}, 2014.

\bibitem{Bramson-Z09}
M.~Bramson and O.~Zeitouni.
\newblock Tightness for a family of recursion equations.
\newblock {\em The Annals of Probability}, 37(2):615--653, 2009.

\bibitem{Bramson:1978}
M.~D. Bramson.
\newblock Minimal displacement of branching random walk.
\newblock {\em Zeitschrift f{\"u}r Wahrscheinlichkeitstheorie und Verwandte
  Gebiete}, 45(2):89--108, 1978.

\bibitem{S-F78}
J.~Fleischman and S.~Sawyer.
\newblock Maximum geographic range of a mutant allele considered as a subtype
  of a brownian branching random field.
\newblock {\em PNAS}, 76(2):872--875, 1979.

\bibitem{Hammersley}
J.~M. Hammersley.
\newblock Postulates for subadditive processes.
\newblock {\em The Annals of Probability}, 2(4):652--680, 1974.

\bibitem{Hu-Shi09}
Y.~Hu and Z.~Shi.
\newblock Minimal position and critical martingale convergence in branching
  random walks, and directed polymers on disordered trees.
\newblock {\em The Annals of Probability}, 37(2):742--789, 2009.

\bibitem{iscoe}
I.~Iscoe.
\newblock A weighted occupation time for a class of measured-valued branching
  processes.
\newblock {\em Probability Theory and Related Fields}, 71(1):85--116, 1986.

\bibitem{Kingman}
J.~F.~C. Kingman.
\newblock The first birth problem for an age-dependent branching process.
\newblock {\em The Annals of Probability}, 3(5):790--801, 1975.

\bibitem{L-P-Z-2014}
S.~P. Lalley, E.~A. Perkins, and X.~Zheng.
\newblock A phase transition for measure-valued sir epidemic processes.
\newblock {\em The Annals of Probability}, 42(1):237--310, 2014.

\bibitem{Lalley-Sellke1987}
S.~P. Lalley and T.~Sellke.
\newblock A conditional limit theorem for the frontier of a branching brownian
  motion.
\newblock {\em The Annals of Probability}, 15(3):1052--1061, 1987.

\bibitem{L-Z2010}
S.~P. Lalley and X.~Zheng.
\newblock Spatial epidemics and local times for critical branching random walks
  in dimensions 2 and 3.
\newblock {\em Probab. Theory Relat. Fields}, 148(3-4):527--566, 2010.

\bibitem{L-Z2011}
S.~P. Lalley and X.~Zheng.
\newblock Occupation statistics of critical branching random walks in two or
  higher dimensions.
\newblock {\em The Annals of Probability}, 39(1):327--368, 2011.

\bibitem{LS15}
S.P. Lalley and Y.~Shao.
\newblock On the maximal displacement of critical branching random walk.
\newblock {\em Probability Theory and Related Fields}, 162(1-2):71--96, 2015.

\bibitem{pinsky}
R.~G. Pinsky.
\newblock On the large time growth rate of the support of supercritical super-
  brownian motion.
\newblock {\em The Annals of Probability}, 23(4):1748--1754, 1995.

\bibitem{spitzer}
F. Spitzer.\newblock{ \em{Principles of Random Walk}.}
\newblock { Springer-Verlag, New York, second addition, 1976.} Graduate Texts in Mathematics, Vol. 34.

\bibitem{Vidmar14}
M.~Vidmar.
\newblock A note on the times of first passage for `nearly right-continuous
  random walks'.
\newblock {\em Electron. Commun. Probab.}, 19(75):1--7, 2014.

\end{thebibliography}
\end{document}